\documentclass[10pt]{article}
\usepackage[a4paper]{geometry}
\usepackage{amsmath, amssymb, amsthm}
\usepackage{appendix}
\usepackage{multicol} 
\usepackage{makeidx}
\usepackage{cite}
\usepackage{captdef}
\makeindex
\usepackage{paralist}
\usepackage{psfrag}
\usepackage{graphicx}
\usepackage{graphics}
\usepackage{color}

\newcommand{\R}{\mathbb{R}}

\newtheorem{theo}{Theorem}[section]
\newtheorem{coro}[theo]{Corollary}
\newtheorem{lemma}[theo]{Lemma}
\newtheorem{prop}[theo]{Proposition}
\theoremstyle{definition}

\newtheorem{remark}[theo]{Remark}

\newcommand{\cI}{{\mathcal I}}

\newcommand{\cM}{{\mathcal M}}   
\newcommand{\cN}{{\mathcal N}}

\newcommand{\cU}{{\mathcal U}}

\renewcommand{\epsilon}{\varepsilon}

\newcommand{\Sn}{{\mathbb S}^{N-1}}

\numberwithin{equation}{section}

\title{Qualitative properties of coexistence and semi-trivial limit profiles of nonautonomous nonlinear parabolic Dirichlet systems}
\author{}
\author{Alberto Salda\~{n}a\footnote{D\'{e}partement de Math\'{e}matique, Universit\'{e} libre de Bruxelles, Bd du Triomphe 1050 Bruxelles. Belgium Campus Plaine, asaldana@ulb.ac.be.}}
\date{}

\begin{document}
\maketitle

\begin{abstract}
We study the symmetry properties of limit profiles of nonautonomous nonlinear parabolic systems with Dirichlet boundary conditions in radial bounded domains. In the case of competitive systems, we show that if the initial profiles satisfy a reflectional inequality with respect to a hyperplane, then all limit profiles are foliated Schwarz symmetric with respect to antipodal points. One of the main ingredients in the proof is a new parabolic version of Serrin's boundary point lemma. Results on radial symmetry of semi-trivial profiles are discussed also for noncompetitive systems.
\end{abstract}
{\footnotesize
\begin{center}
\end{center}
\begin{center}
\end{center}
}

\section{Introduction}

The main motivation for this paper is the study of the asymptotic shape of positive global solutions of the following nonautonomous nonlinear parabolic problem.
\begin{equation}\label{model:competitive:dirichlet}
\begin{aligned}
 (u_i)_t-\Delta u_i&=f_i(t,|x|,u_i)-\alpha_i(|x|,t) u_1u_2, &&\ x\in B,\ t>0,\\
 u_i(x,t)&=0, &&\  x\in \partial B,\ t>0,\\
 u_i(x,0)&=u_{0,i}(x), &&\ x\in B
\end{aligned}
\end{equation}
for $i=1,2.$  Here and in the remainder of the paper, $B$ denotes a ball or an annulus in $\mathbb R^N$ with $N\geq 2.$ The nonlinearity $f$ and the coefficients $\alpha_i$ satisfy some smoothness assumptions.  As an important example of a problem to which our results apply we have the competitive Lotka-Volterra system, that is,
 \begin{equation}\label{Lotka:Volterra:system}
 \begin{aligned}
 (u_1)_t-\Delta u_1 &= a_1(t) u_1 - b_1(t) u_1^2 -\alpha_1(t) u_1u_2,&&\quad x \in B,\ t>0,\\
 (u_2)_t-\Delta u_2 &= a_2(t) u_2 - b_2(t) u_2^2 -\alpha_2(t) u_1u_2,&&\quad x\in B,\ t>0,\\
 u_1=u_2&=0 &&\quad \text{on $\partial B \times (0,\infty)$,}\\
 u_i(x,0)&=u_{0,i}(x)\geq 0&&\quad \text{for $x\in B,$ $i=1,2$,}
 \end{aligned}
 \end{equation}
 where for some $\beta\in(0,1)$ and $i=1,2,$ we have that
 \begin{equation}\label{Lotka:coefficients}
   \begin{aligned}
   a_i,\ b_i,\ \alpha_i \in C^\beta((0,\infty))\cap L^\infty((0,\infty))\qquad\text{ and }\qquad \inf_{t>0} \alpha_i(t)>0.
   \end{aligned}
 \end{equation}

For a discussion on the history and the relevance of system \eqref{Lotka:Volterra:system} we refer to \cite{saldana:weth:systems, cantrell:cosner} and the references therein.
 
To describe the asymptotic shape of a solution $u=(u_1,u_2)$ of \eqref{Lotka:Volterra:system} we use the omega limit set, that is, 
\begin{align}\label{omega:u}
\omega(u)&= \omega(u_1,u_2):=\{(z_1,z_2)\in C(\overline{B})\times C(\overline{B})\::\:  \\&
\lim_{n\to\infty}\|u_1(\cdot,t_n)-z_i\|_{L^\infty(B)}+\|u_2(\cdot,t_n)-z_i\|_{L^\infty(B)}=0 \text{ for some } t_n\to\infty\}.\nonumber
\end{align}
This set is nonempty, compact, and connected for uniformly bounded solutions of \eqref{Lotka:Volterra:system}, see Lemma \ref{regularity:lemma} below. Elements of $\omega(u)$ are called \emph{limit profiles}.  If $z=(z_1,z_2)\in\omega(u)$ is such that $z_1\not\equiv 0$ and $z_2\not\equiv 0$ in $B,$ then we call $z$ a \emph{coexistence limit profile}.  If $z_1\equiv 0$ and $z_2\not \equiv 0$ or if $z_2\equiv 0$ and $z_1\not \equiv 0,$ then we call $z$ a \emph{semi-trivial limit profile}. In this paper we present results on the symmetry of both kinds of limit profiles. The proofs rely on a variant of the moving plane method \textemdash often called \emph{rotating plane method}\textemdash and on stability arguments.

Let us introduce some notation that will be used throughout the paper. Let $\Sn=\{x\in\mathbb R^N: |x|=1\}$ be the unit sphere in
 $\R^N,$ $N\geq 2$.  For a vector $e\in \Sn$, we consider the hyperplane $H(e):=\{x\in \mathbb R^N: x\cdot e=0\}$ and the half domain $B(e):=\{x\in B: x\cdot e>0\}.$  We write also $\sigma_e: \overline{B}\to \overline{B}$ to denote reflection with respect to $H(e),$
 i.e. $\sigma_e(x):=x-2(x\cdot e)e$ for each $x\in B.$ Following \cite{smets}, we say that a function $u\in C(B)$ is \textit{foliated Schwarz symmetric with respect to some unit vector $p\in \Sn$} if $u$ is axially symmetric with respect to the axis $\mathbb R p$ and nonincreasing in the polar angle $\theta:= \operatorname{arccos}(\frac{x}{|x|}\cdot p)\in [0,\pi].$ 

Our main result for system \eqref{Lotka:Volterra:system} is the following.
 
\begin{theo}
Suppose that \eqref{Lotka:coefficients} holds and let $u=(u_1,u_2)$ be a classical solution of \eqref{Lotka:Volterra:system} such that
$u_i\in L^\infty(B \times (0,\infty))$ for
$i=1,2$. Moreover, assume that 
$$
\leqno{\rm (h0)}\quad u_{0,1} \geq u_{0,1} \circ \sigma_e, \:u_{0,2} \leq
u_{0,2}\circ \sigma_e \quad \text{in $B(e)$}\ \
\left \{
\begin{aligned}
&\text{for some $e\in \Sn$ with}\\
&\text{$u_{0,i}\not\equiv u_{0,i}\circ \sigma_e$ for $i=1,2$.}  
\end{aligned}
\right.
$$
Then there is some $p\in \Sn$ such that all elements $(z_1,z_2)\in\omega(u)$ satisfy that $z_1$ is foliated Schwarz symmetric with respect to $p$ and $z_2$ is foliated Schwarz symmetric with respect to $-p.$ 
\end{theo}

This theorem is an direct consequence of the next more general result. In the following $C^{a,a/2}(Q)$ for $a>0$ denotes the standard (parabolic) H\"{o}lder space on $Q\subset R^{N+1}$, see \eqref{holder:norms}. Fix $I_B:=\{|x|\::\: x\in\overline{B}\}.$

\begin{theo}\label{main:theorem}
Let $u_1,u_2\in C^{2,1}(\overline{B}\times(0,\infty))\cap C(\overline{B}\times[0,\infty))\cap L^\infty(B\times(0,\infty))$ be nonnegative functions such that $u=(u_1,u_2)$ is a solution of system \eqref{model:competitive:dirichlet} where the following holds.
\begin{itemize} 
\item[(h1)] For $i=1,2,$ the nonlinearity $f_i:[0,\infty)\times I_B\times [0,\infty)\to \R,$ $(t,r,v)\mapsto f_i(t,r,v)$ is continuously differentiable in $v.$ Further, the functions $f_i$ and $\partial_v f_i$ are H\"{o}lder continuous in $t$ and $r$ for all $v\in [0,\infty),$ and locally Lipschitz continuous in $v$ uniformly with respect to $t$ and $r.$ In other words, there is $\gamma>0$ such that, for every $h\in\{f_i, \partial_v f_i \::\: i=1,2\},$ 
\begin{align*}
h(\cdot,\cdot,v)\in C^{\gamma,\gamma/2}(B\times(0,\infty)) \qquad \text{ for all }v\in[0,\infty)
\end{align*}
and 
\begin{align*}
\sup_{\genfrac{}{}{0pt}{}{\scriptstyle{r\in I_B,
      t>0,}}{\scriptstyle{v,\bar v\in K, v\neq\bar
      v}}}\!\!\!\frac{|h(t,r,v)-h(t,r,\bar v)|}{|v-\bar v|}<
\infty
\end{align*}
for any compact subset $K\subset [0,\infty).$  Moreover $f_i(t,r,0)=0$ for all $r\in I_B,$ $t>0,$ and $i\in\{1,2\}.$
\item[(h2)] There are positive constants $\alpha^*,$ $\alpha_*,$ and $\beta$ such that $\alpha_i\in C^{\beta,\beta/2}(I_B\times(0,\infty))$ and  $\alpha_* \le \alpha_i(r,t)\leq \alpha^*$ for all $r\in I_B,$ $t>0,$ and $i\in\{1,2\}.$
\item[(h3)] There is $e\in \Sn$ for which $u_{0,1}\not\equiv u_{0,1}\circ \sigma_e,$ $u_{0,2}\not\equiv u_{0,2}\circ \sigma_e,$ and $u_{0,1}\geq u_{0,1} \circ \sigma_e, \:u_{0,1} \leq u_{0,2}\circ \sigma_e$ in $B(e).$
\end{itemize}
Then there is some $p\in \Sn$ such that all elements $(z_1,z_2)\in\omega(u)$ satisfy that $z_1$ is foliated Schwarz symmetric with respect to $p,$ and $z_2$ is foliated Schwarz symmetric with respect to $-p.$
\end{theo}

Theorem \ref{main:theorem} is a Dirichlet analog of \cite[Theorem 1.3]{saldana:weth:systems}, where the asymptotic shape of solutions of the Neumann version of system \eqref{Lotka:Volterra:system} is studied, and also foliated Schwarz symmetry with respect to antipodal points is proved.  In both theorems assumption (h3) is essential to perform a rotating plane method and the symmetry result can not hold in general without it; see \cite{saldana:weth:bounded} and \cite{saldana:weth:systems} for a discussion in this regard.

One of the main ingredients in the proofs in \cite{saldana:weth:systems} is an extension of the solution using a reflection with respect to the boundary of the domain. This extension is only possible under Neumann boundary conditions, and in fact, the main perturbation tool \cite[Lemma 3.1]{saldana:weth:systems} can not hold under Dirichlet boundary conditions. 

In this paper, to prove Theorem \ref{main:theorem}, we use a new parabolic version of Serrin's boundary point lemma \cite[Lemma 1]{serrin} that provides quantitative information on the second derivatives at corner points. This information is used to perform a perturbation argument within a rotating plane method scheme. A parabolic version of Serrin's boundary point lemma was obtained in \cite[Lemma 2.1]{alessandri}, where the result from Serrin \cite[Theorem 1]{serrin} is extended to the parabolic setting; however, this version is insufficient for our purposes, since it is crucial to obtain estimates independent of a specific solution, see Lemmas \ref{estimate1} and \ref{estimate2} below.

In this setting, to perform the perturbation argument one needs some equicontinuity of the solution up to second derivatives.  This is the reason why the right hand side in \eqref{Lotka:Volterra:system} and in \eqref{model:competitive:dirichlet} is required to satisfy some H\"{o}lder regularity (this assumption is not needed in \cite{saldana:weth:systems}, but it is a common hypothesis in the study of nonautonomous Lotka-Volterra models, see \cite[Section 5.5]{cantrell:cosner}). With this hypothesis one can guarantee using some standard regularity techniques as in \cite{polacik} and \cite{babin} that the solutions $u_1$ and $u_2$ belong to $C^{2+\gamma,1+\gamma/2}(B)$ for some $\gamma\in(0,1)$ (see Lemma \ref{regularity:lemma} below). 

As already noted in \cite{saldana:weth:systems}, the possibility of the existence of semi-trivial limit profiles presents a complication for the perturbation argument, which depends on lower bounds for some $L^\infty-$norms of solutions to a linearized problem (see Proposition \ref{general:framework:prop} below).  In \cite{saldana:weth:systems} this complication was circumvented using a normalization argument relying again on estimates which only hold for the Neumann problem. Here we adapt this argument to the Dirichlet case using some sharp estimates due to H\'{u}ska, Pol\'{a}\v{c}ik, and Safonov \cite{huska:polacik:safonov} for positive solutions of homogeneous linear problems.

For a class of parabolic cooperative systems, F\"{o}ldes and Pol\'{a}\v{c}ik  \cite{polacik:systems} proved in particular that, if the domain is a ball and the nonlinearity satisfies some monotonicity assumptions, positive solutions are asymptotically radially symmetric and radially decreasing.  Their proofs rely on maximum principles for small domains, which is a different approach from ours and can not be applied here directly (see also Remark \ref{remark:maximum:principles:for:small:domains} in this regard). See also \cite{saldana:weth:bounded} for a version of Theorem \ref{main:theorem} for more general scalar equations using maximum principles for small domains.  

As explained in \cite{saldana:weth:systems,cantrell:cosner}, systems of the type \eqref{model:competitive:dirichlet} are commonly used to model population dynamics. In this setting, the existence of a semi-trivial limit profile can be interpreted as the asymptotic extinction of one of the species.  In such a situation, one may guess that the population density of the remaining species, in the (asymptotic) absence of a competing (or symbiotic) species, is likely to become asymptotically radially symmetric in $B.$  We show in \cite{saldana-phd} that this is indeed the case if additional assumptions are made.  

Let $\lambda_1>0$ denote the first Dirichlet eigenvalue of the Laplacian in $B.$
\begin{theo}[Theorem 5.2 in \cite{saldana-phd}]\label{main:theorem:3} Let $u_1,u_2\in C^{2,1}(\overline{B}\times(0,\infty))\cap C(\overline{B}\times[0,\infty))\cap L^\infty(B\times(0,\infty))$ be nonnegative functions such that $u=(u_1,u_2)$ is a classical solution of 
\begin{equation*}
\begin{aligned}
 (u_i)_t-\Delta u_i &= f_i(u_i) - \alpha_i(x,t) u_1u_2 && \quad \text{ in } B\times(0,\infty),\\
 u_i&=0 &&\quad \text{ on } \partial B\times(0,\infty),
\end{aligned}
\end{equation*}
where $u(\cdot,0)\equiv u_{0,i}\in C(\overline{B})$ is not identically zero, $\alpha_i\in L^\infty(B\times(0,\infty)),$ $f_i\in C^1(\R)$ is strictly concave in $[0,\infty)$, $f_i(0)=0,$ and  $\lim\limits_{s\to\infty}f_i(s)=-\infty$ for $i=1,2.$  If $f'_1(0)>\lambda_1$ and $f'_2(0)>\lambda_1,$ then all the semi-trivial limit profiles of $u$ are radially symmetric. In particular, if there is $(z,0)\in \omega(u),$ then $z$ is radially symmetric and it is the unique positive solution of $-\Delta z = f_1(z)$ in $B$. The analogous claim holds if there is $(0,z)\in \omega(u)$ with $z\not\equiv 0.$
\end{theo}

The proof relies on well-known techniques for nonlinear PDE's with concave nonlinearities as used in \cite{haraux} (see also \cite[Chapters 9,10]{cazenave-haraux}). The hypothesis of Theorem \ref{main:theorem:3} might look a bit restrictive, but they are satisfied by many interesting models; for example, the Lotka-Volterra models for two species. We state this next.

\begin{coro}\label{Lotka:Volterra:corollary} Let $u_1,u_2\in C^{2,1}(\overline{B}\times(0,\infty))\cap C(\overline{B}\times[0,\infty))\cap L^\infty(B\times(0,\infty))$ be nonnegative functions such that $u=(u_1,u_2)$ is a classical solution of
\begin{equation*}
\begin{aligned}
 (u_i)_t-\Delta u_i &= a_i u_i - b_i u_i^2 -\alpha_i(x,t)u_1u_2 &&\quad \text{ in }B\times(0,\infty),\\
 u_i&=0 &&\quad \text{ on }\partial B\times(0,\infty),
\end{aligned}
\end{equation*}
where $u(\cdot,0)\equiv u_{0,i}\in C(\overline{B})$ is not identically zero, $a_i>\lambda_1,$ $b_i>0,$ and $\alpha_i\in L^\infty(B\times(0,\infty))$ for $i=1,2.$  If $(z,0)\in \omega(u)$ then $z\in C(\overline{B})\cap H^1_0(B)$ is a positive radially symmetric function and it is the unique solution of $-\Delta z = a_1 z - b_1 z^2$ in $B.$  The analogous claim holds for $(0,z)\in\omega(u).$
\end{coro}

The Lotka-Volterra model for competing species assumes that the coefficients $\alpha_i$ are positive for $i=1,2,$ but this assumption is not needed for the claim in Corollary \ref{Lotka:Volterra:corollary}. Indeed, Theorem \ref{main:theorem:3} makes \emph{no} assumption on the sign of $\alpha_1$ and $\alpha_2,$ and therefore the result also applies to semitrivial limit profiles of cooperative, competitive, and mixed type systems, which includes, for instance, the predator-prey model.  Also note that the coefficients $\alpha_i$ may depend on $x$ and not only on $|x|$ as in Theorem \ref{main:theorem}.  Te keep this paper short, we do not give here the proof of Theorem \ref{main:theorem:3}. We refer to \cite{saldana-phd} for the proof and other related results.

We now give a brief outline of this paper. Section \ref{linear:equations} is devoted to the parabolic version of Serrin's boundary point lemma for supersolutions of scalar equations and its application to prove a perturbation lemma for solutions to weakly coupled linear systems. In Subsection \ref{subsec:homogeneous:linear:equation} we recollect some estimates from \cite{huska:polacik:safonov}, which will be used for the normalization argument.  In Section \ref{regularity} we explain how the equicontinuity up to second derivatives of any solution is achieved and finally Section \ref{proofs} contains the proof of Theorem \ref{main:theorem}.

To close this introduction, let us mention that, similarly as in \cite{saldana:weth:systems}, the proof of Theorem~\ref{main:theorem} can be adjusted to deal with other kinds of systems. For example, cubic systems, a cooperative version of \eqref{Lotka:Volterra:system}, or irreducible cooperative systems of $n-$equations can be considered. In this last type of systems, no normalization procedure is necessary and therefore also sign-changing solutions are allowed.  For the linearization of these systems one can use the Hadamard formulas similarly as in \cite{polacik:systems}. Then it is relatively easy to adjust the rotating plane method presented here to these problems. Since the statements and the proofs of these results are similar to those presented in \cite{saldana:weth:systems} and in \cite{saldana-phd} for the Neumann case, we do not give any further detail in this regard.  

\vspace{.2cm}

 \noindent \textbf{Acknowledgements:}
We thank Prof. Tobias Weth for very helpful discussions and suggestions related to the paper. This work was partially supported by  DAAD (Germany)-CONACyT (Mexico) and MIS F.4508.14 FNRS (Belgium).
\section{A parabolic version of Serrin's boundary point lemma}\label{linear:equations}

First we fix some notation. Let $B$ be a ball or an annulus in $\R^N$ centered at zero and fix $0 \le A_1 <A_2< \infty$ such that 
\begin{equation}\label{B:definition:dirichlet}
B:= \begin{cases}
\{x\in\mathbb R^N: A_1<|x|<A_2\}, & \text{ if } A_1>0,\\
 \{x\in\mathbb R^N: |x|<A_2\}, & \text{ if } A_1=0.
\end{cases}
\end{equation}
We denote by $B_r(y)\subset\mathbb R^N$ the ball of radius $r$ centered at $y\in\mathbb R^N.$  For a subset $A\subset \overline{B}$ and $\delta>0$ we define
\begin{align}\label{delta:notation}
[A]_{\delta}:=\{x\in \overline{B}\::\: \operatorname{dist}(x,A)\leq \delta\}.
\end{align}

For $\cI\subset \R,$ $e\in \Sn,$ and a function $v:\overline{B}\times \cI\to\R$ we define 
\begin{align*}
v^e:\overline{B}\times \cI\to\R\qquad \text{ by }\qquad v^e(x,t):=v(x,t)-v(\sigma_e(x),t). 
\end{align*}

Fix $I:=[0,1],$ $e_1=(1,0,\ldots,0)\in\R^N,$ and let $v:\overline{B}\times I\to\mathbb [0,\infty)$ be a function such that $v^{e_1}$ satisfies
\begin{equation}
\begin{aligned}\label{ve1:equation}
v^{e_1}_t-\Delta v^{e_1} - c v^{e_1} &\geq 0 \quad &&\text{ in } B({e_1}) \times I,\\
v^{e_1} &= 0 \quad &&\text{ on } \partial B({e_1}) \times I,\\
v^{e_1} &\geq 0 \quad &&\text{ in } B({e_1}) \times I,
\end{aligned}
\end{equation}
and the following holds.
\begin{enumerate}
 \item [$(H_{\alpha,\beta_0})$] There are $\alpha\in(0,1)$ and $\beta_0>0$ such that 
 \begin{align*}
 |v|_{2+\alpha;B\times I}+\|c\|_{L^\infty(B(e_1)\times I)}\leq \beta_0,
 \end{align*}
 where $|\cdot|_{2+\alpha;B\times I}$ denotes the parabolic H\"{o}lder norm, see \eqref{holder:norms}.
 \item [$(H_{k})$] There is $k>0$ such that $\|v^{e_1}\|_{L^\infty(B(e)\times(\frac{1}{7},\frac{4}{7}))} \geq k.$
\end{enumerate}

 \begin{remark}\label{delta:definition}
Let $v\in C^{2,1}(\overline{B}\times I)$ be a function satisfying $(H_{\alpha,\beta_0})$ and $(H_{k}).$  If $v^{e_1}\equiv 0$ on $\partial B(e)\times I,$ then there is $\delta\in(0,\frac{A_1-A_2}{2}),$ depending only on $\alpha,$ $\beta_0,$ $k,$ and $B$ such that $|v^{e_1}|<k$ in $[\partial B(e)]_\delta\times I.$  Therefore we have that $\|v^{e_1}\|_{L^\infty((B(e)\backslash[\partial B(e)]_\delta)\times(\frac{1}{7},\frac{4}{7}))} \geq k.$
\end{remark}

We will use the following version of the parabolic Harnack inequality for general domains as given in \cite{polacik}.
\begin{lemma}[Particular case of Lemma 3.5 in \cite{polacik}]\label{nonhomogeneous:harnack:inequality:polacik} 
Let $\beta_0>0,$ $\varepsilon>0,$ $\delta>0,$ $\theta>0,$ $0<\tau<\tau_1<\tau_2<\tau_3< \tau_4<\infty,$ be some given constants. There are positive constants $\kappa,$ $\kappa_1,$ and $p$ determined only by $B,$ $\beta_0,$ $\delta,$ $\varepsilon,$ $\tau_2-\tau_1,$ $\tau_3-\tau_2,$ $\tau_3-\tau_4,$ and $\theta$ with the following property.  If $D,U$ are subdomains in $B$ with $D\subset \subset U,$ $\operatorname{dist}(\overline{D}, \partial U)\geq \delta,$ $|D|>\varepsilon,$ and $v\in W^{2,1}_{N+1,loc}(U\times (\tau,T))\cap C(\overline{U}\times[\tau,\tau_4])$ is such that $v\geq 0$ in $U\times(\tau,\tau_4)$ and satisfies $v_t - \Delta v- c\,v\geq 0$ in $U\times(\tau,\tau_4),$ where $\tau_1-2\theta\leq \tau\leq \tau_1-\theta$  and $\|c\|_{L^\infty(U\times(\tau,\tau_4))}\leq\beta_0.$  Then
\begin{align*}
 \inf_{D\times (\tau_3,\tau_4)}v \geq \kappa \bigg( \frac{1}{|D\times(\tau_1,\tau_2)|}\int_{D\times(\tau_1,\tau_2)} v^p\ d(x,t) \bigg)^\frac{1}{p}.
\end{align*}
\end{lemma}

Indeed, this follows directly by \cite[Lemma 3.5]{polacik} taking $v\geq 0$ and $g\geq 0.$
 
The following lemma focuses on points of the boundary $\partial B(e_1)$ which are \emph{not} corner points.

\begin{lemma}\label{estimate1}
Let $v\in C^{2,1}(\overline{B}\times I)$ be a function such that $v^{e_1}$ satisfies \eqref{ve1:equation} and such that assumptions $(H_{\alpha,\beta_0})$ and $(H_{k})$ hold. Then given $\delta\in(0,\frac{A_1-A_2}{2}),$ there are positive constants $\varepsilon$ and $\mu$ depending only on $\delta,$ $\alpha,$ $\beta_0,$ $k,$ and $B$ such that
\begin{align}\label{estimate1:2}
v^{e_1}(x,t)\geq \mu x_1\qquad \text{ for all }x\in B(e_1)\backslash [\partial B]_\delta,\ t\in [\,\frac{6}{7}\,,\,1\,]
\end{align}
and 
\begin{align}\label{estimate1:1}
 \frac{\partial v^{e_1}}{\partial \nu}>\varepsilon\qquad \text{ in }\ \ (\partial B(e_1)\backslash[\partial B\cap H(e_1)]_\delta)\times[\,\frac{6}{7}\,,\,1\,].
\end{align}
Here $\nu$ is the inwards unit normal vector field on $\partial B(e_1).$
\end{lemma}
\begin{proof}
 By Remark \ref{delta:definition} there is $\delta_1\in(0,\delta)$ depending only on $\alpha,$ $\beta_0,$ $k,$ $\delta,$ and $B,$ such that 
 \begin{align}\label{delta:definition:eq:1}
\|v^{e_1}\|_{L^\infty((B(e)\backslash[\partial B(e)]_{\delta_1})\times(\frac{1}{7},\frac{4}{7}))} \geq k.  
\end{align}
Let $r< \min\{\delta_1,\frac{1}{7}\}$ be such that for any $t^*\in[\frac{6}{7},1]$ and $x^*\in\partial B(e_1)$ with $\operatorname{dist}(x^*,H(e_1)\cap \partial B)>\delta_1$ we have that $x^*+r \nu(x^*)=:y\in B(e_1)$ and $\partial B_{r}(y,t^*)\cap (\partial B(e_1)\times[0,t^*])=\{(x^*,t^*)\}.$  Fix such a pair $(x^*,t^*)$ and define the sets
 \begin{align*}
  D&:= B_{r}(y,t^*)\cap B_{\frac{{r}}{2}}(x^*,t^*)\cap (B(e_1)\times [0,t^*]),\\
\Gamma_1&:= \partial B_{r}(y,t^*)\cap \overline{B_{\frac{{r}}{2}}(x^*,t^*)}\cap (\overline{B(e_1)}\times [0,t^*]),\\
\Gamma_2&:= \overline{B_{r}(y,t^*)}\cap \partial{B_{\frac{{r}}{2}}(x^*,t^*)}\cap (\overline{B(e_1)}\times [0,t^*]),
 \end{align*}
and the function $z:D\to [0,1]$ given by
\begin{align*}
z(x,t):= (e^{-\gamma(|x-y|^2+(t-t^*)^2)}-e^{-\gamma {r}^2})e^{-\beta_0t} \quad \text{ with } \gamma=\frac{2(N+1)}{{r}^2}.
\end{align*}

An easy calculation shows that
\begin{align}
&z_t-\Delta z -cz \nonumber\\
&=e^{-\gamma (|x-y|^2+(t-t^*)^2)-\beta_0t}2\gamma(-2\gamma |x-y|^2+N-(t-t^*))+(-\beta_0-c)z\nonumber\\
&\leq e^{-\gamma (|x-y|^2+(t-t^*)^2)-\beta_0t}\gamma(-\gamma{r}^2+2(N+1))= 0\qquad \text{in $D.$}\label{z:subsolution}
\end{align}

Let $0<d<\min\{\operatorname{dist}(\Gamma_2,\partial B(e_1)\times[\frac{5}{7},1])\ ,\ \delta_1\}$ and $K:=\{x\in B(e_1)\::\: \operatorname{dist}(x,\partial B(e_1))\geq d\}.$  Note that $\Gamma_2\subset K\times[\frac{5}{7},1]$ and $K$ only depends on $\alpha,$ $\beta_0,$ $k,$ $\delta,$ and $B.$ Then by $(H_{\alpha,\beta_0}),$ \eqref{delta:definition:eq:1}, and Lemma \ref{nonhomogeneous:harnack:inequality:polacik}, there is $\mu_1=\mu_1(\alpha,\beta_0,\delta,k,B)>0$ such that $v^{e_1}\geq \mu_1$ in $K\times[\frac{5}{7},1].$  In particular, we have that $v^{e_1}\geq \mu_1$ in $\Gamma_2.$   Let $w:D\to\R$ be given by $w:=v^{e_1}-\mu_1 z.$ Since $z\equiv 0$ in $\Gamma_1$ and $v^{e_1}\geq0$ in $B(e_1)$ we have that $w\geq 0$ in $\Gamma_1$ and, since $z\leq 1$ in $\Gamma_2,$ it follows that $w \geq 0$ on $\Gamma_1\cup\Gamma_2.$ Further, by \eqref{ve1:equation} and \eqref{z:subsolution} we get that $ w_t-\Delta w -cw \geq 0$ in $D.$  Thus the parabolic maximum principle implies that $w\geq 0$ in $D.$ Since $w(x^*,t^*)=0$ we have that $\frac{\partial w(x^*,t^*)}{\partial \nu}\geq 0$ and thus
 \begin{align*}
  \frac{\partial v^{e_1}}{\partial \nu}(x^*,t^*)\geq \mu_1\frac{\partial z}{\partial \nu}(x^*,t^*)\geq 2e^{-\gamma{r}^2-\beta_0}\mu_1\gamma{r}=:\varepsilon>0.
 \end{align*}
Since $\delta>\delta_1,$ this proves \eqref{estimate1:1}. Moreover, since $\frac{\partial v^{e_1}}{\partial e_1}\geq\varepsilon$ in $(H(e_1)\backslash [\partial B]_{\delta_1})\times[\frac{6}{7},1],$ there is, by $(H_{\alpha,\beta_0}),$ some $\delta_2=\delta_2(\alpha,\beta_0,\delta,k,B)\in(0,\delta_1)$ such that 
  \begin{align}\label{estimate1:3}
   \frac{\partial v^{e_1}}{\partial e_1}\geq\frac{\varepsilon}{2} \qquad \text{ in }\quad  ([H(e_1)\backslash [\partial B]_{\delta_1}]_{\delta_2})\times[\frac{6}{7},1].
  \end{align}
  
On the other hand, $B(e_1)\backslash [\partial B(e_1)]_{\delta_1}\subset B(e_1)\backslash[\partial B(e_1)]_{\delta_2}$ since $\delta_2<\delta_1.$ Thus, by $(H_{\alpha,\beta_0}),$ \eqref{delta:definition:eq:1}, and Lemma \ref{nonhomogeneous:harnack:inequality:polacik}, there is $\mu_2>0$ depending only on $\alpha,$ $\beta_0,$ $\delta,$ $k,$ and $B$ such that $v^{e_1}\geq \mu_2$ in $(B(e_1)\backslash[\partial B(e_1)]_{\delta_2})\times[\frac{5}{7},1].$ This, together with \eqref{estimate1:3} easily implies \eqref{estimate1:2} for some $\mu>0$ depending only on $\alpha,$ $\beta_0,$ $\delta,$ $k,$ and $B.$
 \end{proof}

We now turn our attention to the \emph{corner points} on the boundary $\partial B(e_1)$.
 
\begin{lemma}\label{estimate2}
Let $v\in C^{2,1}(\overline{B}\times I)$ be a function such that $v^{e_1}$ satisfies \eqref{ve1:equation} and such that assumptions $(H_{\alpha,\beta_0})$ and $(H_{k})$ hold. Then there is $\varepsilon>0$ depending only on $\alpha,$ $\beta_0,$ $k,$ and $B$ such that
\begin{align}\label{estimate2:estimate}
 \frac{\partial^2 v^{e_1}(x,1)}{\partial s^2}>\varepsilon \qquad  \text{ and }\qquad  \frac{\partial^2 v^{e_1}(x,1)}{\partial \tilde s^2}<-\varepsilon
\end{align}
 for all $x\in\partial B\cap H(e_1)$, where $s=\frac{1}{\sqrt{2}}(\nu+e_1)\in\Sn,$ $\tilde s=\frac{1}{\sqrt{2}}(-\nu+e_1)\in\Sn,$  and $\nu$ is the inwards unit normal vector field on $\partial B.$
 \end{lemma}
\begin{proof}
 Let $x^*\in \partial B\cap H(e_1)$ and $0<r<\min\{\frac{1}{7}, \frac{A_2-A_1}{4}\},$ see \eqref{B:definition:dirichlet}. Then $x^*+r \nu(x^*)=:y\in H(e_1)\cap B$ and $\partial B_{r}({y},1)\cap (\partial{B}\times I) =\{(x^*,1)\}.$  Define
 \begin{align*}
 U&:= B_{r}({y},1)\cap B_{\frac{r}{2}}(x^*,1)\cap (B(e_1)\times I),\\
\Lambda_1&:= \partial B_{r}({y},1)\cap \overline{B_{\frac{r}{2}}(x^*,1)}\cap (\overline{B(e_1)}\times I),\\
\Lambda_2&:= \overline{B_{r}({y},1)}\cap \partial{B_{\frac{r}{2}}(x^*,1)}\cap (\overline{B(e_1)}\times I),\\
\Lambda_3&:= \overline{B_{r}({y},1)}\cap \overline{B_{\frac{r}{2}}(x^*,1)}\cap ({H(e_1)}\times I),
 \end{align*}
and the function $\varphi:\overline{B}\times I\to [-A_2,A_2]$ given by 
\begin{align*}
\varphi(x,t):= x_1(e^{-\theta(|x-{y}|^2+(t-1)^2)}-e^{-\theta{r}^2})e^{-\beta_0t}\quad \text{ with }\theta=\frac{2(N+3)}{r^2}.
\end{align*}

A direct calculation shows that
\begin{align}
&\varphi_t-\Delta \varphi -c\varphi \nonumber\\
&=e^{-\theta (|(x,t)-(y,1)|^2)-\beta_0t}2\theta x_1(-2\theta |x-y|^2+N+2-t)-(\beta_0+c)\varphi\nonumber\\
&\leq e^{-\theta (|(x,t)-(y,1)|^2)-\beta_0t}2\theta x_1(-2\theta (\frac{r}{2})^2+N+3)= 0\qquad \text{in $U.$  }\label{subsolution:1}
\end{align}

Since $\operatorname{dist}(\Lambda_2,\partial B\times I)>0,$ there is $\delta_1=\delta_1(B)>0$ such that $\Lambda_2\ \subset \ (\overline{B(e_1)}\backslash [\partial B]_{\delta_1})\times[\,\frac{6}{7}\,,\,1\,].$  Then, by Lemma \ref{estimate1} with $\delta=\delta_1$, there is $\mu>0$ depending only on $\alpha,$ $\beta_0,$ $k,$ and $B$ such that $v^{e_1}(x,t)\geq \mu x_1$ for $x\in (\overline{B(e_1)}\backslash [\partial B]_{\delta_1})$ and $t\in[\,\frac{6}{7}\,,\,1\,].$  In particular $v^{e_1}(x,t)\geq \mu x_1$ for $(x,t)\in\Lambda_2.$ Define the function $\psi:\overline{B}\times I\to\R$ by $\psi(x,t):=v^{e_1}(x,t)-\mu\varphi(x,t).$  Then $\psi \geq 0$ on $\Lambda_2,$ because $\varphi(x,t)\leq x_1$ for $(x,t)\in\Lambda_2$.  Moreover, since $v^{e_1}\geq 0$ in $B(e_1)\times I$ and $\varphi\equiv 0$ in $\Lambda_1\cup \Lambda_3$ by definition, we get that $\psi \geq 0$ on $\Lambda_1\cup \Lambda_2\cup \Lambda_3.$ Further, by \eqref{ve1:equation} and \eqref{subsolution:1}, $\psi_t-\Delta \psi -c\psi\geq 0$ in $U.$ Thus the parabolic maximum principle implies that $\psi\geq 0$ in $U.$  Now, remember that $s(x^*)=\frac{1}{\sqrt{2}}(\nu(x^*)+e_1)$. By direct calculation $\frac{\partial \varphi}{\partial s}(x^*,1)=0$ and, since $v^{e_1}\equiv 0$ on $\partial B(e_1)\times I,$ we have that 
\begin{align*}
\frac{\partial v^{e_1}}{\partial s}(x^*,1)=\frac{1}{\sqrt{2}}(\frac{\partial v^{e_1}}{\partial e_1}(x^*,1)+\frac{\partial v^{e_1}}{\partial \nu}(x^*,1))=0.
\end{align*}
This implies that $\frac{\partial \psi}{\partial s}(x^*,1)=0=\psi(x^*,1).$ Since $\psi\geq 0$ in $U,$ it follows that  $\frac{\partial^2 \psi}{\partial s^2}(x^*,1)\geq 0,$ and therefore
\begin{align*}
 \frac{\partial^2 v^{e_1}}{\partial s^2}(x^*,1)\geq \mu\frac{\partial^2 \varphi}{\partial s^2}(x^*,1)&=-4\mu e^{-\theta r^2-\beta_0}\theta\frac{1}{\sqrt{2}}(x-{y})\cdot s=2\mu e^{-\theta r^2-\beta_0}\theta r,
  \end{align*}
 which yields the first inequality in \eqref{estimate2:estimate} with $\varepsilon:=2\mu e^{-\theta r^2}\theta r>0$.  For the second inequality, note that the function $v^{e_1}$ is antisymmetric in $x$ with respect to $H(e_1),$ and therefore 
\begin{align*}
\frac{\partial^2 v^{e_1}}{\partial \tilde s^2}(x^*,1)=\frac{\partial^2 v^{e_1}}{\partial (-\tilde s)^2}(x^*,1)= -\frac{\partial^2 v^{e_1}}{\partial s^2}(x^*,1)\leq -\varepsilon,
\end{align*}
where $\tilde s(x^*)=\frac{1}{\sqrt{2}}(-\nu(x^*)+e_1).$ The proof is finished.
\end{proof}

The next lemma will be helpful to guarantee positivity near corner points.

\begin{lemma}\label{positive:at:corners}
Let $w\in C^2(\overline{B(e_1)})$ be a function such that the following holds.
\begin{itemize}
 \item[(i)] $w\equiv 0$ on $\partial B(e_1).$
 \item[(ii)] There is $\varepsilon>0$ such that 
\begin{align*}
 \frac{\partial^2 w(x)}{\partial s(x)^2}>\varepsilon \qquad  \text{ and }\qquad  \frac{\partial^2 w(x)}{\partial \tilde s(x)^2}<-\varepsilon\qquad \text{ for all } x\in H(e_1)\cap \partial B,
\end{align*}
where $s(x)=\frac{1}{\sqrt{2}}(\nu(x)+e_1),$ $\tilde s(x)=\frac{1}{\sqrt{2}}(-\nu(x)+e_1),$  and $\nu$ is the inwards unit normal vector field on $\partial B.$
\item[(iii)] There is a function $\chi :[0,\operatorname{diam}(B)\,] \to  [0,\infty)$ with $\lim \limits_{\vartheta \to 0} \chi(\vartheta)=0$ such that $|D^2 w(x)-D^2 w(y)|\leq \chi(|x-y|)$ for all $x,y\in \overline {B(e_1)}.$
\end{itemize}
Then there exists $\delta>0$ depending only on $\varepsilon,$ $\chi,$ and $B$ such that $w\geq 0$ in $[H(e_1)\cap \partial B]_{\delta}\cap B(e_1).$
\end{lemma}
\begin{proof}
By hypothesis $(ii)$ and $(iii)$ there is $\delta_1=\delta_1(\varepsilon,\chi)>0$ such that, for every $x^*\in H(e_1)\cap \partial B$ we have that
\begin{align}\label{perturbed:second:derivatives}
\frac{\partial^2 w(x)}{\partial s(x^*)^2}>\varepsilon \quad  \text{ and }\quad  \frac{\partial^2 w(x)}{\partial \tilde s(x^*)^2}<-\varepsilon\quad  \text{ for all }x\in \overline{B_{\delta_1}(x^*)\cap B(e_1)}.
\end{align}
If $A_1>0$ \textemdash where $A_1$ is as in \eqref{B:definition:dirichlet}\textemdash we also assume that 
\begin{align}\label{touches:the:boundary}
\delta_1<\bigg(\frac{1}{\sin(\frac{\pi}{4})}-1\bigg)A_1.
\end{align}

We show that the claim holds for any $\delta\in(0,\frac{\delta_1}{2}].$ Without loss of generality we may assume for the rest of the proof that the domain $B$ is two-dimensional, since otherwise we can repeat the following argument to $w$ restricted to $B(e_1)\cap P$ where $P$ is any plane containing $\R e_1$ to yield the result. 

By contradiction, assume there are $x^*\in H(e_1)\cap \partial B$ and $y\in B(e_1)\cap B_\delta(x^*)$ such that $w(y)<0.$ Then there are $y_1\in H(e_1)\cap B_{\delta_1}(x^*)$ and $\lambda_1>0$ such that $y=y_1+\lambda_1 \tilde s(x^*).$  Let $L_1:=\{\lambda\geq 0 \::\: y_1+\lambda \tilde s(x^*)\in \overline{B(e_1)\cap B_{\delta_1}(x^*)}\}$ and $A:=\{y_1+\lambda \tilde s(x^*)\::\: \lambda\in L_1\}.$  Note that $A\cap \partial B\neq \emptyset$, by our choice of $\delta$ and \eqref{touches:the:boundary}. In particular, we may find $\lambda_B>0$ such that $y_1+\lambda_B \tilde s(x^*)\in\partial B.$  Moreover, due to the fact that $B$ is assumed to be two-dimensional, there is $\lambda_2>0$ such that $x^*+\lambda_2s(x^*)=:y_2\in A.$  Let $L_2:=\{\lambda\geq 0\::\: x^*+\lambda s(x^*)\in \overline{B(e_1)\cap B_{\delta_1}(x^*)}\}$  and define the functions $f_1:L_1\to\R$ and $f_2:L_2\to\R$ by
\begin{align*}
 f_1(\lambda):= w(y_1+\lambda \tilde s(x^*)), \qquad f_2(\lambda):= w(x^*+\lambda  s(x^*)).
\end{align*}
By \eqref{perturbed:second:derivatives}, we have that $f_2''>\varepsilon$ in $L_2.$ By Assumption $(i)$ it follows that
\begin{align*}
 f_2'(0)=\frac{\partial w(x^*)}{\partial s(x^*)}=\frac{1}{\sqrt 2}\bigg(\frac{\partial w(x^*)}{\partial \nu(x^*)}+\frac{\partial w(x^*)}{\partial e_1}\bigg)=0=w(x^*)=f_2(0).
\end{align*}
Therefore $f_2(\lambda)>0$ for $\lambda\in L_2\cap(0,\infty).$ Since $y_2\in A,$ there is $\lambda_3>0$ such that $f_1(\lambda_3)=w(y_2)=f_2(\lambda_2)>0.$  But then $f_1(0)=0,$ $f_1(\lambda_1)=w(y)<0,$ $f_1(\lambda_3)>0,$ and $f_1(\lambda_B)=0.$ This contradicts the fact that $f_1''<-\varepsilon$ in $L_1$ by \eqref{perturbed:second:derivatives}. Therefore $w\geq 0$ in $B_\delta(x^*)\cap B(e_1)$ as claimed.
\end{proof}

We are ready to prove a perturbation result for supersolutions of scalar equations.
\begin{lemma}\label{perturbationlemma:dirichlet}
Let $v\in C^{2,1}(\overline{B}\times I)$ be a function such that $v^{e_1}$ (resp. $-v^{e_1}$) satisfies \eqref{ve1:equation} and such that assumptions $(H_{\alpha,\beta_0})$ and $(H_{k})$ hold.  Then there exists $\rho>0$ depending only on $\alpha,$ $\beta_0,$ $k$, and $B$ such that $v^{e'}(\cdot,1)\geq 0$ (resp. $-v^{e'}(\cdot,1)\geq 0$) in $B(e')$ for all $e'\in \Sn$  with $|e_1-e'|<\rho$.
\end{lemma}
\begin{proof} 
Assume that $v^{e_1}$ satisfies \eqref{ve1:equation} and that $(H_{\alpha,\beta_0})$ and $(H_{k})$ hold. By Lemma~\ref{estimate2} and assumption $(H_{\alpha,\beta_0}),$ there are $\varepsilon_1>0$ and $\rho_1>0$ such that 
\begin{align*}
\frac{\partial^2 v^{e'}}{\partial s(x)^2}(x,1)> \varepsilon_1\quad  \text{ and }\quad \frac{\partial^2 v^{e'}}{\partial \tilde s(x)^2}(x,1)<- \varepsilon_1 \quad\text{ for } x\in \partial B\cap H(e')
\end{align*}
and for $e'\in\Sn$ with $|e'-e_1|<\rho_1.$ By Lemma~\ref{positive:at:corners}, there is $\delta_1\in(0,\frac{A_1-A_2}{2})$ such that 
\begin{align}\label{easy3}
 v^{e'}(\cdot,1)\geq 0\qquad  \text{ in } [\partial B\cap H(e')]_{\delta_1}\cap B(e')
\end{align}
for all $e'\in\Sn$ with $|e'-e_1|<\rho_1.$  Further, by Lemma~\ref{estimate1} with $\delta=\delta_1$ there is $\varepsilon_2>0$ with
$\frac{\partial v^{e_1}}{\partial \nu}(\cdot,1)> \varepsilon_2$ on $\partial B(e_1)\backslash[\partial B\cap H(e_1)]_{\delta_1}.$ Then, by $(H_{\alpha,\beta_0}),$ there are $\delta_2\in(0,\delta_1)$ and $\rho_2\in(0,\rho_1)$ with the property $\frac{\partial v^{e'}}{\partial \nu}(\cdot,1)> \varepsilon_2$ in $[\partial B(e')\backslash[\partial B\cap H(e')]_{\delta_1}]_{\delta_2}$ for all $e'\in\Sn$ with $|e'-e_1|<\rho_2.$  Since $v^{e'}(\cdot,1)\equiv 0$ on $\partial B(e')$ we have
\begin{align}\label{easy6}
v^{e'}(\cdot,1)>0\qquad \text{ in }[\partial B(e')\backslash[\partial B\cap H(e')]_{\delta_1}]_{\delta_2}\cap B(e')
\end{align}
for all $e'\in\Sn$ with $|e'-e_1|<\rho_2.$  Moreover, by Lemma~\ref{estimate1} with $\delta=\delta_2$ there is $\mu>0$ with $v^{e_1}(x,1)\geq \mu x_1>\frac{\mu \delta_2}{2}$ for all $x\in B(e_1)\backslash [\partial B(e_1)]_{\delta_2}.$ Then, by $(H_{\alpha,\beta_0}),$ there is $\rho_3\in(0,\rho_2)$ such that 
\begin{align}\label{easy5}
v^{e'}(x,1)>\frac{\mu\delta_2}{2}\qquad \text{ for all }x\in B(e')\backslash [\partial B(e')]_{\delta_2}
\end{align}
for all $e'\in\Sn$ with $|e'-e_1|<\rho_3.$  Therefore \eqref{easy3}, \eqref{easy6}, and \eqref{easy5} imply that $v^{e'}(\cdot,1)\geq 0$ in $B(e')$ for all $e'\in\Sn$ with $|e'-e_1|<\rho_3.$ It is easy to check that $\rho_3$ depends only on $\alpha,$ $\beta_0,$ $k$, and $B$ so this yields the claim for $v^{e_1}$. 

From the proof of Lemmas \ref{estimate1} and \ref{estimate2} it is clear that their claims also hold if we write $-v^{e_1}$ instead of $v^{e_1}.$ Therefore we can argue exactly the same if we now assume that $-v^{e_1}$ satisfies \eqref{ve1:equation}. This ends the proof.
\end{proof}

We now use Lemma \ref{perturbationlemma:dirichlet} to prove a perturbation result for systems. First we fix some notation. For a pair of functions $(v_{1},v_{2})\in C^{2,1}(\overline{B}\times I)\cap C(\overline{B}\times I)\times C^{2,1}(\overline{B}\times I)\cap C(\overline{B}\times I)$ and $e\in \Sn$ define the functions $v_1^e,v_2^e:\overline{B}\to\R$ by
\begin{equation}\label{linearization:v:e}
 v_1^e(x,t):=v_1(x,t)-v_1(\sigma_e(x),t),\qquad  v_2^e(x,t):=v_2(\sigma_e(x),t)-v_2(x,t).
\end{equation}
Assume that 
 \begin{equation}\label{systems:general}
 \begin{aligned}
 (v_1^{e_1})_t-\Delta v_1^{e_1} -c_{11} v_1^{e_1}-c_{12} v_2^{e_1}&\geq 0 \qquad \text{ in } B({e_1})\times I,\\
 (v_2^{e_1})_t-\Delta v_2^{e_1} -c_{21} v_1^{e_1}-c_{22} v_2^{e_1}&\geq 0 \qquad \text{ in } B({e_1})\times I,\\
v_1^{e_1} = v_2^{e_1} &=0\qquad \text{ on } \partial B({e_1})\times I,\\
 v_1^{e_1}\geq 0\quad\text{ and }\quad v_2^{e_1}&\geq 0\qquad \text{ in } B({e_1})\times I,
 \end{aligned}
 \end{equation}
where 
\begin{align}\label{coefficients:general}
 c_{12}\geq 0\quad  \text{and}\quad c_{21}\geq 0\qquad \text{ in }B({e_1})\times I.
\end{align}

To simplify the notation, we write
\begin{align}\label{Q:notation}
Q:=B({e_1})\times[\frac{1}{7},\frac{4}{7}]\ \  \text{ and }\ \  Q_\delta:=(B(e_1)\backslash [\partial B(e_1)]_\delta)\times [\frac{2}{7}\,,\,\frac{3}{7}],
\end{align}
for $\delta>0,$ where $[\,\cdot\,]_\delta$ was defined in \eqref{delta:notation}. Note that $Q_\delta\subset\subset Q.$

\begin{prop}\label{general:framework:prop}
 Let $\gamma\in(0,1)$ and $M>0$ be given constants and let $v_{1},v_{2}\in C^{2,1}(\overline{B}\times I)\cap C(\overline{B}\times I)$ be functions such that 
 \begin{align}\label{equicontinuity:general}
 |v_{i}|_{2+\gamma;B\times I}< M \qquad  \text{ for }i=1,2.
 \end{align}
 Suppose $(v_1^{e_1},v_2^{e_1})$ satisfies \eqref{systems:general}, \eqref{coefficients:general}, and $\sum\limits_{i,j=0}^2\|c_{ij}\|_{L^\infty(B(e_1)\times I)}<M.$ 
 \begin{itemize}
 \item [(i)] If there is $\delta>0$ and $\varepsilon_1>0$ such that
 \begin{align}\label{claim1:general}
\|v_{2}^{e_1}\|_{L^\infty(Q_\delta)} > \varepsilon_1\qquad \text{ and }\qquad \inf\limits_{Q_\delta}c_{12}>\varepsilon_1,
 \end{align}
then there exists $\sigma>0$ depending only on $\varepsilon_1,$ $\delta,$ $M,$ $\gamma,$ and $B$ such that $\|v_{1}^{e_1}\|_{L^\infty(Q)} >\sigma,$ where $Q$ and $Q_\delta$ are as in \eqref{Q:notation}.
\item[(ii)] If there is $\varepsilon_2>0$ such that $\|v_{i}^{e_1}\|_{L^\infty(Q)} > \varepsilon_2$ for $i=1,2,$ then there is $\rho>0$ depending only on $\varepsilon_2,$ $M,$ $\gamma,$ and $B$ such that 
\begin{align*}
\text{$v_{i}^{e'}(\cdot,1)\geq 0$ in $B(e')$}\qquad\text{ for $i=1,2$ and $e'\in \Sn$ with $|{e_1}-e'|<\rho.$}
\end{align*}
\end{itemize}
\end{prop}
\begin{proof}
To prove $(i),$ note that \eqref{claim1:general} and \eqref{equicontinuity:general} imply that there is a nonempty open subset $\Omega\subset Q_\delta$ such that $v_{2}^{e_1} > \varepsilon_1$ in $\Omega$ and $|\Omega|\geq \sigma_1$ for some $\sigma_1$ depending only on $\varepsilon_1,$ $\delta,$ $\gamma,$ $M,$ and $B.$ Let $\varphi\in C_c^\infty(Q)$ be a function depending only on $\delta$ and $B$ such that $\varphi\geq 0$ in $Q$ and $\varphi\equiv 1$ in $Q_\delta$.  Then  
 \begin{align*}
  A:=\int_Q c_{12}\, v_i^{e_1}\, \varphi\, d(x,t)\geq \int_{\Omega} c_{12}\, v_2^{e_1}\, d(x,t)\geq \varepsilon_1^2 \sigma_1,
  \end{align*}
by \eqref{coefficients:general}, \eqref{claim1:general}, and the choice of $\varphi$.  On the other hand, 
\begin{align*}
 A &= \int_{Q} [(v^{e_1}_1)_t-\Delta v_1^{e_1}-c_{11}v_1^{e_1}]\varphi\ d(x,t)=\int_{Q} (-\varphi_t-\Delta \varphi-c_{11}\varphi)v_1^{e_1}\ d(x,t)\\
 &\leq\|v_1^{e_1}\|_{L^\infty(Q)} \int_{Q} |\varphi_t|+|\Delta \varphi|+M|\varphi|\ d(x,t).
\end{align*}
by \eqref{systems:general} and integration by parts. Therefore $\|v_1^{e_1}\|_{L^\infty(Q)}\geq \frac{ \varepsilon_1^2\sigma_1}{C}=:\sigma>0,$ where $C:=\int_{Q} |\varphi_t|+|\Delta \varphi|+M|\varphi|\ d(x,t).$ Since $\sigma$ depends only $\varepsilon_1,$ $\delta,$ $\gamma,$ $M,$ and $B$ this proves the first claim $(i)$. For the second claim $(ii)$, note that
 \begin{equation*}
(v_i)^{e_1}_t-\Delta v_i^{e_1} -c_{ii} v_i^{e_1}\geq c_{ij} v_j^{e_1}\geq 0 \qquad \text{ in } B({e_1})\times I,\text{ for } i,j=1,2\ i\neq j,
\end{equation*}
by \eqref{systems:general} and \eqref{coefficients:general}.  Then, for $i=1,2$ we have that $v_i$ satisfies the assumptions of Lemma \ref{perturbationlemma:dirichlet} with $\alpha=\gamma,$ $k=\varepsilon,$ and $\beta_0=M.$ The application of Lemma \ref{perturbationlemma:dirichlet} gives a constant $\rho>0$ such that $v_{i}^{e'}(\cdot,1)\geq 0$ in $B(e')$ for $i=1,2$ and $e'\in \Sn$ with $|{e_1}-e'|<\rho,$ as claimed.
 \end{proof}

\subsection{Homogeneous linear equations}\label{subsec:homogeneous:linear:equation}

We quote some estimates from \cite{huska:polacik:safonov}. This will be used in the proof of Theorem \ref{main:theorem} to guarantee condition \eqref{claim1:general} for the linearized system.

\begin{lemma}[Particular cases of Lemma 3.9 and Corollary 3.10 in \cite{huska:polacik:safonov}] \label{huska:polacik:safonov:lemma}
Let  $v\in C^{2,1}(B\times (0,\infty))\cap C(\overline{B}\times [0,\infty))$ be a positive solution of
\begin{equation*}
\begin{aligned}
 v_t-\Delta v -c v=0\quad \text{ in } B\times (0,\infty),\qquad  v=0 \quad\text{ on } \partial B\times (0,\infty),
\end{aligned}
\end{equation*}
where $\|c\|_{L^\infty(B\times (0,\infty))}<\beta_0$ for some $\beta_0>0.$  Then, there are positive constants $C_1,$ $C_2,$ and $\vartheta$ depending only on $B$ and $\beta_0$ such that
\begin{align}\label{huska:quotient:0}
\frac{v(x,\tau)}{\|v(\cdot,\tau)\|_{L^\infty(B)}}\geq C_1 (\operatorname{dist}(x,\partial B))^\vartheta, \qquad x\in B,\ \tau\in[1,\infty),
\end{align}
and
\begin{align}\label{huska:quotient:1}
\frac{\|v(\cdot,\tau+t)\|_{L^\infty(B)}}{\|v(\cdot,\tau)\|_{L^\infty(B)}}\in\bigg [C_2\,,\,\frac{1}{C_2}\bigg], \qquad  \tau\in(1,\infty),\ t\in[0,1],
\end{align}
where $\operatorname{dist}(x,\partial B):=\inf\{|x-y|\::\: y\in \partial B\}.$
\end{lemma}

By combining \eqref{huska:quotient:0} and \eqref{huska:quotient:1} we get the following.

\begin{coro}\label{huska:polacik:safonov:corollary}
Assume the hypothesis of Lemma \ref{huska:polacik:safonov:lemma}.  For any $k>0$ there are positive constants $C$ and $\vartheta$ depending only on $B,$ $\beta_0,$ and $k$ such that, for any $\tau\geq 2k,$
\begin{align*}
C\geq \frac{v(x,\tau+t)}{\|v(\cdot,\tau)\|_{L^\infty(B)}}\geq C^{-1} (\operatorname{dist}(x,\partial B))^\vartheta, \qquad x\in B,\ t\in[-k,k].
\end{align*}
\end{coro}
\begin{proof}
 Let $\tilde v: B\times(0,\infty)\to\R$ be given by $\tilde v(x,t):= v(x,kt)$ for $x\in B,$ $t>0.$ Clearly $\tilde v$ satisfies the assumptions of Lemma \ref{huska:polacik:safonov:lemma}. Let $C_1,$ $C_2,$ and $\vartheta$ be the constants (depending only on $B,$ $\beta_0,$ and $k$) given by Lemma \ref{huska:polacik:safonov:lemma} for $\tilde v.$  Note that \eqref{huska:quotient:1} implies that 
\begin{align}\label{huska:quotient:2}
\frac{\|\tilde v(\cdot,\tau)\|_{L^\infty(B)}}{\|\tilde v(\cdot,\tau+t)\|_{L^\infty(B)}}\in\bigg [C_2\,,\,\frac{1}{C_2}\bigg], \qquad  \tau\in(1,\infty),\ t\in[0,1],
\end{align}

Then, by \eqref{huska:quotient:0}, \eqref{huska:quotient:1}, and \eqref{huska:quotient:2}, 
\begin{align*}
 C_2^{-1}\|\tilde  v(\cdot,\tau)\|_{L^\infty(B)} &\geq \tilde v(x,t+\tau)\geq C_1 (\operatorname{dist}(x,\partial B))^\vartheta\|\tilde  v(\cdot,t+\tau)\|_{L^\infty(B)}\\
 &\geq C_1 C_2 (\operatorname{dist}(x,\partial B))^\vartheta \|\tilde  v(\cdot,\tau)\|_{L^\infty(B)}
\end{align*}
for all $x\in B,$ $\tau\geq 2,$ and $t\in(-1,1).$  The result follows.
\end{proof}

\section{Regularity of solutions}\label{regularity}
Let $\alpha\in(0,1],$ $k$ a nonnegative integer, $a=k+\alpha,$ and $\Omega\subset \R^N$ be a domain. Put $Q:=\Omega\times(\tau,T)$ for $0\leq \tau<T.$ Following \cite[page 4]{quittner-souplet} we define $C^{a,a/2}(Q):=\{f : |f|_{a;Q}<\infty \},$ where
\begin{equation}\label{holder:norms}
\begin{aligned}
&[f]_{\alpha;Q}:= \sup\bigg\{ \frac{|f(x,t)-f(y,s)|}{|x-y|^\alpha+|t-s|^\frac{\alpha}{2}}\::\: (x,t),(y,s)\in Q,\ (x,t)\neq (y,s)\bigg\},\\
&|f|_{a;Q}:=\sum_{|\beta|+2j\leq k}\sup_{Q}|D_x^\beta D_t^j f|+\sum_{|\beta|+2j= k}[D_x^\beta D_t^j f]_{\alpha;Q},
 \end{aligned}
 \end{equation}
and $D^\beta_x,D^j_t$ denote spatial and time derivatives of order $\beta\in \mathbb N_0^N$ and $j\in\mathbb N_0$ respectively.

The following is the Dirichlet version of \cite[Lemma 2.1]{saldana:weth:systems} complemented with standard regularity arguments (see for example \cite[Remark 48.3 and Remark 47.4 (iii)]{quittner-souplet} or \cite[Theorem 2.2]{babin}) to guarantee $C^{2+\gamma,1+\gamma/2}-$regularity if the right hand side of the equation is H\"{o}lder continuous.

\begin{lemma}\label{regularity:lemma}
Let $K>0$ be a given constant, $\Omega \subset \mathbb R^{N}$ be a smooth bounded domain, $I\subset \R$ open, $g\in L^\infty(\Omega\times I),$ and let $v\in C^{2,1}(\overline{\Omega}\times I)\cap C(\overline{\Omega\times I})$ be a classical solution of
\begin{equation*}
\begin{aligned}
v_t-\Delta v =g(x,t)  \quad \text{in $\Omega \times I$},\qquad v=0  \quad  \text{ on } \partial \Omega \times I,
\end{aligned}
\end{equation*}
where $\|v\|_{L^\infty(\Omega\times I)}+\|g\|_{L^\infty(\Omega\times I)}<K.$  Let ${\cal I}\subset I$ with $\operatorname{dist}({\cal I},\partial I)\geq 2.$  Then there are constants $\tilde C>0$ and $\tilde \gamma\in(0,1),$ depending only on $\Omega$ and $K$ such that 
\begin{align*}
|v|_{1+\tilde\gamma;\overline{\Omega}\times[s,s+2]}\leq \tilde C \qquad \text{ for all }s\in\cI.
\end{align*}
Moreover, if $|g|_{\alpha;\Omega\times I}<K$ for some $\alpha>0,$ then there are constants $C>0$ and $\gamma\in(0,1),$ depending only on $\Omega,$ $K,$ and $\alpha$ such that 
\begin{align*}
|v|_{2+\gamma;\overline{\Omega}\times[s+1,s+2]}\leq C \qquad \text{ for all }s\in\cI.
\end{align*}

In particular, if $I=(0,\infty)$, then the semiorbit $\{v(\cdot,t)\::\: t \geq 1\}$ is relatively compact in $C(\overline{\Omega}).$ Therefore $\omega(v)$ is a nonempty connected compact subset of $C(\overline{\Omega})$ satisfying  $\lim\limits_{t\to\infty}\inf\limits_{z\in\omega(v)}\|v(\cdot,t)-z\|_{L^\infty(\Omega)}=0.$
\end{lemma}
\begin{proof}
Fix $s\in \cI$ and set $Q:=\overline{\Omega}\times [s,s+2]\ \subset \ \overline{\Omega\times I}. $  Then $v\in W^{2,1}_{N+3}(Q)$ because $v\in C^{2,1}(Q).$ Then \cite[Theorem 7.30, p.181]{lieberman} implies the existence of $C_1>0$ depending only on $\Omega$ and $K$ such that
\begin{align*}
 \|D^2 v\|_{L^{N+3}(Q)}+\|v_t\|_{L^{N+3}(Q)}&\leq C_1(\|g\|_{L^{N+3}(Q)}+\|u\|_{L^{N+3}(Q)})\leq 2C_1|Q|^{\frac{1}{N+3}}K.
\end{align*}
Then, by a standard interpolation argument, $\|v\|_{W_{N+3}^{2,1}(Q)}\leq C_2$ for some constant $C_2(\Omega,K)=C_2>0.$  By Sobolev embeddings (see, for example, \cite[embedding (1.2)]{quittner-souplet} and the references therein), we then have that $v\in C^{1+\tilde\gamma,(1+\tilde\gamma)/2}(Q)$ for $\tilde\gamma=1-\frac{N+2}{N+3}\ \in\ (0,1),$ and there is a constant $C_3(\Omega)=C_3>0$ such that
\begin{align}\label{holder1}
 |v|_{1+\gamma;Q}\leq C_3 \|v\|_{W_{N+3}^{2,1}(Q)}\leq C_3 C_2.
 \end{align}

Now, assume $|g|_{\alpha;\Omega\times I}<K$ for some $\alpha>0.$ Let $\gamma=\min\{\alpha,\tilde\gamma\}$ and define $w:B\times I\to\R$ by $w(x,t):=(t-s)v(x,t).$ It follows that $w\equiv 0$ on $\partial_P Q$ \textemdash here $\partial_P$ denotes the parabolic boundary\textemdash and $w_t(x,t)-\Delta w(x,t) = (t-s)g(t,x)-v(x,t)$ in $Q.$  Then, by \cite[Theorem 48.2]{quittner-souplet} (see also \cite[Theorem 4.28]{lieberman} and \cite[Section 1.6.6]{cantrell:cosner}), there is a constant $C_4>0$ depending only on $\Omega$ and $\alpha$ such that
\begin{align*}
 |w|_{2+\gamma;Q}\leq C_4(\|w\|_{L^\infty(Q)}+|tg-v|_{\gamma;Q}).
\end{align*}
Then, by \eqref{holder1}, there is some constant $C_5>0$ depending only on $K,$ $\Omega,$ and $\alpha$ such that $ |w|_{2+\gamma;Q}\leq  C_5.$ This implies that $|v|_{2+\gamma;\overline{\Omega}\times [s+1,s+2]}\leq C$ for some constant $C>0$ depending only on $K,$ $\Omega,$ and $\alpha.$  Finally, the last claim follows from the Arzel\`{a}-Ascoli Theorem (see \cite[Proposition 53.3]{quittner-souplet}). 
\end{proof}

 \section[Dirichlet parabolic systems]{Dirichlet parabolic systems}\label{proofs}

For this section let 
\begin{align}\label{unif:bounded:u:i}
u_1,u_2\in C^{2,1}(\overline{B}\times(0,\infty))\cap C(\overline{B}\times[0,\infty))\cap L^\infty(B\times(0,\infty)) 
\end{align}
be nonnegative functions such that $u=(u_1,u_2)$ is a classical solution of \eqref{model:competitive:dirichlet} and such that assumptions (h1) and (h2) from Theorem \ref{main:theorem} are fulfilled. Note that $u_i$ satisfies $(u_i)_t-\Delta u_i = g_i(x,t)$ in $B\times(0,\infty)$ with $g_i(x,t):=f_i(t,|x|,u_i(x,t))-\alpha_i(|x|,t)u_1(x,t)u_2(x,t)$ for $i=1,2$. Moreover, by (h1),(h2), and \eqref{unif:bounded:u:i} there is $K>0$ such that
\begin{align*}
\|u_i\|_{L^\infty(B \times (0,\infty))}+\|g_i\|_{L^\infty(B \times (0,\infty))}< K \qquad \text{ for }i=1,2.
\end{align*}
 Then Lemma~\ref{regularity:lemma} implies that $|u_i|_{a;B\times[s,s+1]}<C$ for all $s\in [1,\infty)$, $i=1,2$ and for some $a\in(1,2)$ and $C>0$.  But this implies that $g_i$ is H\"{o}lder continuous in $B\times[1,\infty)$ for $i=1,2.$ Then, by Lemma \ref{regularity:lemma}, there are $\gamma\in(0,1)$ and $M>0$ such that 
\begin{align}\label{equicontinuity::of:u:i}
 |u_i|_{2+\gamma;B\times[s,s+1]}<M \quad \text{ for all }s\in [2,\infty),\ i=1,2.
\end{align}

In particular the semiorbits $\{u_i(\cdot,t):t\geq 1\}$ are precompact in $C(\overline{B})$ for $i=1,2$,  and
\begin{align}\label{unif:convergence}
 \lim_{t\to\infty}\inf_{(z_1,z_2)\in\omega(u)}\|u_1(\cdot,t)-z_1\|_{L^\infty(B)}+\|u_2(\cdot,t)-z_2\|_{L^\infty(B)}=0,
\end{align}
where $\omega(u)$ is as in \eqref{omega:u}.  By a standard compactness argument, the omega limit set $\omega(u)$ of $u=(u_1,u_2)$ and the omega limit set $\omega(u_i)$ are related as follows:
\begin{equation}
  \label{eq:3:neumann}
\omega(u_i)= \{z_i \::\: z=(z_1,z_2) \in \omega(u)\} \qquad \text{for $i=1,2$.}   
\end{equation}

\subsection{Linearization}

The linearization of \eqref{model:competitive:dirichlet} is the same as in \cite[Section 5]{saldana:weth:systems}. We recall it here for completeness. For $e\in\Sn$ define $(u_1^e,u_2^e)$ as in \eqref{linearization:v:e}. The same notation is used if the functions do not depend on time, that is, for a pair $z=(z_1,z_2)$ of functions $z_i: \overline{B} \to\R,$ $i=1,2$, we set $z_1^e(x):=z_1(x)-z_1(\sigma_e(x))$ and $z_2^e(x):=z_2(\sigma_e(x))-z_2(x)$ for all $ x\in \overline{B},\ t>0.$  

Since $u=(u_1,u_2)$ solves \eqref{model:competitive:dirichlet}, for fixed $e \in\Sn$ we have 
\begin{equation}\label{linear:dirichlet}
\begin{aligned}
(u^e_1)_t-\Delta u_1^e &=c^e_{11} u_1^e+c^e_{12} u_2^e \quad &&\text{ in }B(e)\times(0,\infty),\\
(u^e_2)_t-\Delta u_2^e &=c^e_{21} u_1^e+c^e_{22} u_2^e\quad &&\text{ in }B(e)\times(0,\infty),\\
u_1^e=u_2^e&= 0 \quad &&\text{ on } \partial B(e)\times(0,\infty),
\end{aligned}
\end{equation}
where $c_{ij}^e\in L^\infty(B\times(0,\infty))$ are given for  $i,j=1,2,$ $i\neq j,$ $x\in B,$ and $t>0$ by $c^e_{ij}(x,t):=\alpha_i(|x|,t) u_i(x,t)\geq 0$ and $c^e_{ii}(x,t):=\hat c_i^e(x,t)- \alpha_i(|x|,t) u_j(\sigma_e(x),t)$ with
\begin{align*}
\hat c_i^e(x,t):=\int_0^1 \partial_v f_i(t,|x|,su_i(x,t)+(1-s)u_i(\sigma_e(x),t))ds.
\end{align*}
Here $\partial_v f_i$ stands for the derivative of $(t,r,v)\mapsto f_i(t,r,v)$ with respect to $v$.  Moreover, by (h1),(h2), and \eqref{unif:bounded:u:i} there is $M\geq 1$ such that
\begin{equation}\label{M:definition}
\sum_{i,j=1}^2\|c^e_{ij}\|_{L^\infty(B\times(0,\infty))}\leq {M}\qquad \text{for all $e\in \Sn.$  }
\end{equation}

\subsection{Proof of Theorem \ref{main:theorem}}

To prove Theorem \ref{main:theorem} we use the following geometric characterization of foliated Schwarz symmetry from \cite{saldana:weth:systems}, which is inspired by a result of Brock \cite{brock:2003}.

\begin{coro}[Corollary 2.6 of \cite{saldana:weth:systems}]\label{sec:symm-char}
 Let $\cU\subset C(\overline{B})$ and suppose that the set 
 \begin{align*}
  \cM:=\{e\in \Sn \::\: z(x) \ge z(\sigma_e(x)) \text{ for all }x\in B(e) \text{ and } z \in \cU\}
 \end{align*}
contains a nonempty subset $\cN$ with the following properties
 \begin{itemize}
 \item[(i)] $\cN$ is relatively open in $\Sn$;
\item[(ii)] for every $e \in \partial \cN$ and $z \in \cU$ we have $z \le z \circ \sigma_e$ in $B(e)$. Here $\partial \cN$ denotes the relative boundary of $\cN$ in $\Sn$. 
 \end{itemize}
Then there is $p \in \Sn$ such that all elements of $\cU$ are foliated Schwarz symmetric with respect to $p$.
\end{coro}

Let $\cN:=\{e\in \Sn : u^e_i> 0\  \text{ in }  B(e)\times [T,\infty)\text{ for }i=1,2 \text { and some }T>0\}.$

\begin{lemma} \label{M:open:neumann:dirichlet}
The set $\cN$ is relatively open in $\Sn$. 
\end{lemma}
\begin{proof} Let $e\in \cN.$ Without loss we may assume that $e=e_1.$  Then $(u_1^{e_1},u_2^{e_1})$ is a solution of \eqref{linear:dirichlet}, and there is $T>0$ such that $u^{e_1}_i> 0$ in $B({e_1})\times(T,\infty).$ Let $M$ satisfy \eqref{M:definition} and \eqref{equicontinuity::of:u:i}.  Then, Proposition \ref{general:framework:prop} $(ii)$ applied to $v_i : B\times [0,1]\to [0,\infty)$ given by $v_i(x,t):=u_i(x,t-T)$ for $i=1,2$ yields the existence of $\rho>0$ such that $u_i^{e'}(\cdot,T+1)\geq 0$ in $B(e')$ for $e'\in\Sn$ with $|e'-{e_1}|<\rho.$ Since $u_i^{e'}$ also satisfies \eqref{linear:dirichlet}, the maximum principle for systems (see, e.g., \cite{protter}) implies $u_i^{e'}>0$ in $B(e')\times(T+1,\infty)$ and therefore $e'\in \cN$ for all $e'\in\Sn$ with $|e'-e|<\rho.$ Thus $\cN$ is open. 
\end{proof}

\begin{lemma}\label{normalization:lemma:2:dirichlet}
For every $e\in \partial \cN$ and every $z \in \omega(u)$ we have $z^e_1 \equiv z^e_2 \equiv 0$ in $B(e)$.
\end{lemma}
\begin{proof}
Without loss we can assume that $\partial N\neq \emptyset.$ Let $z=(z_1,z_2) \in\omega(u).$  We will only show that $z_2^e\equiv 0$ in $B(e)$ for all $e \in \partial \cN$, since the same
argument can be adjusted to show that $z_1^e\equiv 0$ in $B(e)$ for all $e\in \partial \cN$.  Arguing by contradiction, assume there is $\hat e\in\partial \cN$ and $k>0$ such that 
\begin{align}\label{contradiction:hypothesis1}
 \|z_2^{\hat e}\|_{L^\infty(B(\hat e))}>k.
\end{align}

We now use a normalization argument to construct a suitable pair of functions that satisfies the assumptions of Proposition \ref{general:framework:prop} (i) and (ii). Let $t_n\to\infty$ be an increasing sequence with $t_1>6$ and such that $u_i(\cdot,t_n)\to z_i$ uniformly in $\overline B$ for $i=1,2.$ Define $I_n:=[t_n-3,t_n+3] \subset \R$ and $\beta_n:=\|u_1(\cdot,t_n)\|_{L^\infty(B)}$ for $n\in\mathbb N,$ and the functions $v_n : \overline{B} \times I_n \to \R$ given by $v_n(x,t):= \frac{u_1(x,t)}{\beta_n}$for $n\in\mathbb N.$ By Corollary \ref{huska:polacik:safonov:corollary} there exists $\eta>1$ and $\vartheta>0$ such that 
\begin{equation}
  \eta \geq v_n \geq \frac{1}{\eta}\operatorname{dist}(x,\partial B)^\vartheta\qquad \text{on $B\times I_n$}\quad
 \text{for all 
   $n\in\mathbb N.$} \label{eta:dirichlet}
\end{equation}
Moreover, we have that $|v_n|_{1+\tilde \gamma;{B}\times[s,s+2]}< \tilde C$ for all $s\in[t_n-2,t_n+2],$ $n\in\mathbb N$ and for some $\tilde \gamma\in(0,1)$ and $\tilde C>0.$  This follows from Lemma~\ref{regularity:lemma} and the fact that $v_n$ satisfies
\begin{equation}\label{preregularity}
\begin{aligned}
(v_n)_t-\Delta v_n &= (c - \alpha_1 u_2)v_n\ \  \text{ in }B\times I_n,\quad v_n=0 \ \ \text{in $\partial B \times I_n,$}
\end{aligned}
\end{equation}
where $c(x,t):=\int_0^1 \partial_v f_1(t,|x|,su_1(x,t))ds$ for $x\in B$ and $t>0.$  Note that $c \in C^{\beta,\beta/2}(B\times [2,\infty))$ for some $\beta>0$ by our assumptions in (h1) and \eqref{equicontinuity::of:u:i}.  This implies that the right hand side of \eqref{preregularity} is H\"{o}lder continuous. Then, by Lemma \ref{regularity:lemma} and \eqref{equicontinuity::of:u:i}, there are $\gamma\in(0,1)$ and $M_1>0$ such that
\begin{align}\label{hoeldercontinuos:1:dirichlet}
 |u_2|_{2+\gamma;{B}\times[s,s+1]}+|v_n|_{2+\gamma;{B}\times[s,s+1]}< M_1
\end{align}
for all $s\in[t_n-1,t_n+1],$ $n\in\mathbb N.$ For $e \in \Sn$ and $n \in \mathbb N$ we also consider $v_n^e : \overline{B} \times I_n \to \R$ given by $v_n^e(x,t):= v_n(x,t)-v_n(\sigma_e(x),t),$ and we note that 
\begin{equation}\label{normalized:equations:dirichlet}
\begin{aligned}
(v_n^e)_t-\Delta v_n^e- c_{11}^e v_n^e &=\tilde c_{12} u_2^e &&\qquad \text{in $B(e) \times I_n$,}\\
(u_2^e)_t-\Delta u_2^e- c_{22}^e u_2^e &=\tilde c_{21}^e v^e_n &&\qquad \text{in $B(e) \times I_n$,}\\
v_n^e=u_2^e &= 0 &&\qquad \text{on $\partial B(e) \times I_n$,}
\end{aligned}
\end{equation}
where $\tilde c_{12}:= \frac{c_{12}^e}{\beta_n}=\alpha_1 v_n,$ $\tilde c_{21}^e:= c_{21}^e \beta_n=c_{21}^e\|u_1(\cdot,t_n)\|_{L^\infty(B)},$ and the coefficients $c_{ij}^e$ are as in \eqref{linear:dirichlet} for $i,j=1,2.$  Then, by \eqref{M:definition}, \eqref{eta:dirichlet}, (h1), and (h2), there is $M\geq M_1$ such that 
\begin{equation}\label{coefficients:bound}
\begin{aligned}
 \|c_{11}^e\|_{L^\infty(B(e)\times I_n)}&+\|c_{22}^e\|_{L^\infty(B(e)\times I_n)}\\
 &+\|\tilde c_{12}^e\|_{L^\infty(B(e)\times I_n)}+\|\tilde c_{21}^e\|_{L^\infty(B(e)\times I_n)}\leq M.
\end{aligned}
\end{equation}
for all $e\in \Sn.$ By \eqref{contradiction:hypothesis1}, \eqref{unif:convergence}, \eqref{hoeldercontinuos:1:dirichlet}, and Remark \ref{delta:definition}, there is $\delta>0$ such that, passing to a subsequence of $t_n,$ we have
\begin{align}\label{contradiction:hypothesis2}
 \|u_2^{\hat e}(\cdot,t_n)\|_{L^\infty(B(\hat e)\backslash[\partial B(\hat e)]_\delta}>k\qquad \text{ for all }n\in\mathbb N.
\end{align}

Moreover, by assumption (h2) and \eqref{eta:dirichlet}, there is $\varepsilon_1\in(0,k)$ such that
\begin{align}\label{c21:equation}
 \tilde c_{12}=\alpha_1 v_n>\frac{\alpha_*}{\eta}\operatorname{dist}(x,\partial B)^\vartheta>\varepsilon_1 \qquad \text{ in }\quad (B(e)\backslash[\partial B(e)]_\delta)\times I_n
\end{align}
for all $e\in \Sn$ and $n\in\mathbb N$ with $\delta$ as in \eqref{contradiction:hypothesis1}. For the choices of $M,$ $\delta,$ $\gamma,$ and $\varepsilon_1$ made in this proof, let $\sigma>0$ be the constant given by Proposition \ref{general:framework:prop} $(i).$  Without loss we may assume that $\sigma\leq k.$ Moreover, let $\rho>0$ be the constant given by Proposition \ref{general:framework:prop} $(ii)$ with $\varepsilon_2=\sigma$ and $M,$ $\delta,$ $\gamma,$ as above.

By \eqref{hoeldercontinuos:1:dirichlet}, \eqref{contradiction:hypothesis2}, and since $\hat e\in\partial \cN,$ there is $\bar e\in\cN$ with $|\bar e-\hat e|<\rho/2$ and
\begin{align}\label{perturbed:function}
 \|u_2^{\bar e}(\cdot,t_n)\|_{L^\infty(B(\bar e)\backslash [\partial B(\bar e)]_\delta}>k\qquad \text{ for all }n\in\mathbb N.
\end{align}

Because $\bar e\in\cN$ and $t_n\to\infty,$ there is $\bar n\in\mathbb N$ such that $u_2^{\bar e}>0$ and $v_{\bar n}^{\bar e}>0$ in $B(\bar e)\times I_{\bar n}.$  Taking into account \eqref{hoeldercontinuos:1:dirichlet}, \eqref{normalized:equations:dirichlet}, and \eqref{coefficients:bound}, we see that the functions $w_1,w_2:B\times[0,1]\to\infty$ given by $w_1(x,t):=v_{\bar n}(x,t_{\bar n}-\frac{2}{7}+t)$ and $w_2(x,t):=u_2(x,t_{\bar n}-\frac{2}{7}+t),$ satisfy the assumptions of Proposition \ref{general:framework:prop}. Then, by \eqref{c21:equation} and \eqref{perturbed:function},
Proposition \ref{general:framework:prop} $(i)$ yields that $\|v_{\bar n}^{\bar e}\|_{L^\infty(B(\bar e)\times [t_{\bar n}-\frac{1}{7},t_{\bar n}+\frac{2}{7}])}>\sigma.$  This together with \eqref{perturbed:function} implies, by Proposition \ref{general:framework:prop} $(ii),$ that $u_2^{e'}(\cdot,t_{\bar n}+\frac{6}{7})\geq 0$ and $v_{\bar n}^{e'}(\cdot,t_{\bar n}+\frac{6}{7})=\frac{u_1^{e'}(\cdot,t_{\bar n}+\frac{6}{7})}{\beta_{\bar n}}\geq 0,$ for all $e'\in\Sn$ with $|e'-\bar e|<\rho.$  Then the maximum principle for parabolic systems (see, e.g., \cite[Ch. 3 Sec. 8 Thm. 13]{protter}) implies that $u_2^{e'}>0$ and $u_1^{e'}>0$ in $B(e')\times(t_{\bar n}+\frac{6}{7},\infty),$ and therefore $e'\in\cN$ for all $e'\in\Sn$ with $|e'-\bar e|<\rho.$  But this yields in particular that $\hat e\in\cN,$ which contradicts the fact that $\hat e\in\partial\cN,$ since $\cN$ is open by Lemma \ref{M:open:neumann:dirichlet}. Therefore $z_2^{\hat e}\equiv 0$ in $B(\hat e)$ as claimed and this ends the proof.
\end{proof}

\begin{proof}[Proof of Theorem \ref{main:theorem}]
Define $\cU:=\omega(u_1) \cup -\omega(u_2) = \{z_1,-z_2\::\: z \in \omega(u)\},$ where the last equality is a consequence of (\ref{eq:3:neumann}). Also let $\cM:=\{e\in \Sn \::\:z^e \ge 0 \quad\text{ in } B(e)  \text{ for all } z\in \cU\}.$  Then we have that $\cN\subset \cM.$ Moreover, for $e\in \Sn$ as in (h3), we have $u_i^e(\cdot,0)\geq 0,\; u_i^e(\cdot,0)\not\equiv 0$ in $B(e)$ for $i=1,2.$  The maximum principle then implies that $u_i^e> 0$ on $B(e) \times (0,\infty)$ for $i=1,2$, so that $e \in \cN$ and thus $\cN$ is nonempty. Moreover $\cN$ is a relatively open subset of $\Sn$ by Lemma \ref{M:open:neumann:dirichlet} and, by Lemma \ref{normalization:lemma:2:dirichlet}, $z\equiv z\circ\sigma_e$ for all $z\in\cU$ and $e\in\partial\cN.$ The result now follows from Corollary~\ref{sec:symm-char}. 
\end{proof}

\begin{remark}\label{remark:maximum:principles:for:small:domains}
 The upper bound in \eqref{eta:dirichlet} holds only in the intervals $I_n=[t_n-3,t_n+3]$ for $n\in\mathbb N.$  In particular, the norm $\|v_n(\cdot,s_n)\|_{L^\infty(B)}$ could tend to infinity for some sequence $s_n\to\infty$ (this would occur, e.g., if $u_1(\cdot,t_n)\to 0$ but $u_1(\cdot,s_n)\not\to 0$ as $n\to\infty$). In this case the coefficient $\tilde c_{12}$ in the linearization \eqref{normalized:equations:dirichlet} can not be uniformly bounded independently of $n\in\mathbb N.$ This is the reason why the \emph{eventual} positivity of the perturbed difference functions\textemdash $(v^{e'}_n,u_2^{e'})$ in the proof of Lemma \ref{normalization:lemma:2:dirichlet}\textemdash is crucial.  An approach based on maximum principles for small domains as in \cite{polacik:systems} would have the advantage of requiring weaker regularity assumptions on the nonlinearity, but this technique typically only guarantees \emph{asymptotic} positivity of $(v^{e'}_n,u_2^{e'})$ and therefore, to perform a perturbation argument, one needs that the coefficients of the linearization are uniformly bounded \emph{independently} of $(v^{e'}_n,u_2^{e'})$ for all times, which, as explained above, does not happen in general.  For scalar equations, eventual positivity using maximum principles for small domains was obtained in \cite[Proposition 6.11]{huska}, but an extension of this result to the case of linear systems remains, up to our knowledge, an open question. 
\end{remark}

\bibliographystyle{plain}
\bibliography{dirichlet}

\end{document}